\documentclass[11pt]{amsart}

\usepackage{amsmath,amsxtra,amssymb,amsthm,amsfonts}
\usepackage{amsmath,
} 
\usepackage{ams math}
\usepackage{amssymb}
\usepackage{amsthm}
\usepackage{dsfont}
\usepackage{upgreek}
\usepackage{scalerel}
\usepackage{enumitem}
\usepackage{hyperref}
\usepackage{array}
\usepackage{ams fonts}
\usepackage{graphics}

\usepackage{eps fig}
\usepackage{caption}
\usepackage{tikz}
\usepackage{pdfpages}
\usepackage[utf8]{inputenc}
\usepackage[english]{babel}
 \usepackage{comment}
\newtheorem{theorem}{Theorem}[section]
\newtheorem*{theorem*}{Theorem}
\newtheorem*{remark*}{Remark}
\newtheorem{lemma}[theorem]{Lemma}
\newtheorem{definition}[theorem]{Definition}
\newtheorem{remark}[theorem]{Remark}
\newtheorem{prop}[theorem]{Proposition}
\newtheorem{example}[theorem]{Example}

\newtheorem{corollary}[theorem]{Corollary}

\newcommand{\tr}{tr}
\DeclareMathOperator{\Hs}{H}
\DeclareMathOperator{\ml}{\ensuremath{\mathds{L}}}
\DeclareMathOperator{\mr}{\ensuremath{\mathds{R}}}

\newcommand{\ssubset} {\!\!\subset\! \!}
\newcommand{\Rc}{Rc}
\newcommand{\Rcu}{\ensuremath{\Rc_{\text{\fontsize{2.5}{4}\selectfont {U}}}}}
\newcommand{\Du}{\ensuremath{\Delta_{\text{\fontsize{2.5}{4}\selectfont {U}}}}}
\newcommand{\Dv}{\ensuremath{\Delta_{\text{\fontsize{2.5}{4}\selectfont {V}}}}}

\newcommand{\Zu}{\ensuremath{Z_{\text{\fontsize{2.5}{4}\selectfont {U}}}}}
\newcommand{\Zv}{\ensuremath{Z_{\text{\fontsize{2.5}{4}\selectfont {V}}}}}
\newcommand{\Hess}{\mbox{Hess}}

\DeclareMathOperator{\Mz}{\ensuremath{\Bbb{M}}}

\DeclareMathOperator{\dom}{dom}
\DeclareMathOperator{\id}{id}
\DeclareMathOperator{\Ent}{Ent}
\DeclareMathOperator{\lind}{\tiny{Lind}}
\DeclareMathOperator{\mfsi}{\ensuremath{M_fSI}}
\DeclareMathOperator{\cfsi}{\ensuremath{C_fSI}}
\DeclareMathOperator{\clsi}{\ensuremath{CLSI}}
\DeclareMathOperator{\cpsi}{\ensuremath{C_pSI}}
\DeclareMathOperator{\mpsi}{\ensuremath{M_pSI}}
\DeclareMathOperator{\mlsi}{\ensuremath{MLSI}}
\DeclareMathOperator{\cN}{\ensuremath{\mathcal{N}}}
\DeclareMathOperator{\triple}{\ensuremath{(\mathcal{N}\ssubset\mathcal{M},\delta,\tau)}}
\DeclareMathOperator{\fix}{fix}

\newcommand{\hhz}{\vspace{0.3cm}}
\usepackage[margin=1.0 in]{geometry}
\usepackage{listings}
\begin{document}
\title{Complete Sobolev Type Inequalities}
\author[H. Li]{Haojian  Li}
\address{Department of Mathematics\\
University of Illinois, Urbana, IL 61801, USA}
\email[Haojian Li]{hli102@illinois.edu}
\maketitle
\begin{abstract}  We establish
Sobolev type inequalities in the noncommutative settings
by generalizing monotone metrics in the space of quantum states, such as matrix-valued Beckner inequalities. We also discuss examples such as random transpositions and  Bernoulli-Laplace models.
\end{abstract}
\section{Introduction}
Poincar\'e inequalities (PIs) and  log Sobolev inequalities (LSIs) have  been well developed in the last few decades. See \cite{ledoux99, gz03} for properties, applications and criterion of PIs and LSIs. Gross showed that log Sobolev inequalities and hypercontractivity are equivalent for Dirichlet form operators, see \cite{gross}.  Beckner inequalities (BIs), as an interpolation between PIs and LSIs, were introduced by Beckner (\cite{Beck89}) in 1989 for the canonical Gaussian measures on $\Bbb{R}^{n}$. Later Ledoux (\cite{ledoux97}) introduced a family of inequalities with the same pattern of LSIs and PIs, which also solved the regularity issues of porous medium equations (\cite{demange, vazquez, bgl13}). 
\hhz

Log Sobolev inequalities in the quantum (noncommutative) settings have been studied recently, see
\cite{CM, CM2, LJL, LJR, DR20, BCR20}. The idea of characterizing matrix-valued Sobolev type inequalities is still absent from the literature. Surprisingly, we explore a vast variety of Sobolev type inequalities by introducing the generalized monotone metrics in the space of quantum states.
\hhz

Let us first recall that an ergodic system $T_{t}=e^{-t\Delta}$ on a probability space $(\Omega, \mu)$ satisfies the $\lambda$-LSI  if there exists $\lambda>0$ such that 
\begin{align} \label{lsi}
\int \rho^{2} \ln(\rho^{2}) d\mu -\int \rho^{2} d\mu \ln(\int \rho^{2} d\mu)\leq \frac{1}{\lambda} \mathcal{E}_{\Delta}(\rho,\rho)
\end{align}
for any function $\rho$, where $\mathcal{E}_{\Delta}(\rho,\sigma)=\int_{\Omega} \Delta (\rho) \sigma d\mu
$ is the energy form.  We now use the notation $\Ent(\rho)=\int \rho \ln(\rho) d\mu -\int \rho d\mu \ln(\int \rho d\mu)$ for the relative entropy. 
 An equivalent formulation of LSIs is the exponential decay of  the relative entropy:
\begin{align}\label{lsie}
\Ent(T_{t}(\rho))\leq e^{-\lambda t}\Ent(\rho)
\end{align}
for any positive $\rho$, see \cite{bgl13, LJL}. For a convex function $f$,  let us consider the relative entropy functional 
$$\Ent_{f}(\rho)=\int f(\rho)-f(\sigma)-(\rho-\sigma)f'(\sigma) d\mu,$$
where $\sigma=\int \rho d\mu$. We observe that $\Ent_{f}=\Ent$ for $f(x)=x\ln(x)$. Then $T_{t}$ satisfies the generalized Sobolev inequality associated to $f$ if there exists $\lambda>0$ such that
\begin{align} \label{si}
\Ent_{f}(T_{t}(\rho))\leq e^{-\lambda t} \Ent_{f}(\rho)
\end{align}
for any positive $\rho$.
Again $f(x)=x\ln(x)$ returns the classical LSIs. Let $p\in(1,2)$ and $f(x)=x^{p}$, then \eqref{si} is equivalent to 
\begin{align*}
\|\rho\|_{p}^{p}-\|\rho\|_{1}^{p}\leq \frac{p}{\lambda(p)} \mathcal{E}_{\Delta}(\rho,\rho^{p-1}),
\end{align*}
where $p\mathcal{E}_{\Delta}(\rho,\rho^{p-1})$ is the analogue of the Fisher information associated to $\Delta$. The limiting cases $p\to 1^{+}$ and $p\to 2^{+}$ reduce to LSIs and PIs, respectively (\cite{bt06}).  
Setting $q=\frac{2}{p}$ and $g=\rho^{1/q}$, we obtain BIs
\begin{align}\label{Beck} \|g\|_{2}^{2}-\|g\|_{q}^{2}\leq \frac{4-2q}{\lambda(2/q)}\mathcal{E}_{\Delta}(g,g),
\end{align}
which was first introduced Beckner (\cite{Beck89}) in 1989 for the canonical Gaussian measure on $\Bbb{R}^{n}$ with optimal constants $\lambda(q)=2$. 
\hhz

We aim at extending \eqref{si} to a finite von Neumann algebra $(\mathcal{N},\tau)$ equipped with a normal faithful tracial state $\tau$.  
We consider the semigroup $T_{t}=e^{-tA}: \mathcal{N}\to \mathcal{N}$ of  completely positive self-adjoint unital maps. Let $\mathcal{N}_{\fix}=\{\rho | T_{t}(\rho)=\rho\}$ be the fixed point algebra of $T_{t}$, which admits the conditional expectation.
Then the generator $A$ is said to satisfy the $\lambda$-modified $f$-Sobolev inequality ($\mfsi$) if 
$$d^{f}(T_{t}(\rho)\|E(\rho))\leq e^{-\lambda t} d^{f}(\rho\|E(\rho)),\quad \forall \rho\in\mathcal{N}_{+},$$ 
where $d^{f}(\rho\|\sigma)=\tau\left(f(\rho)-f(\sigma)-(\rho-\sigma)f'(\sigma) \right)$ for $\rho,\sigma\in\mathcal{N}_{+}$. We say $A$ satisfies $\lambda$-complete $f$-Sobolev inequality ($\cfsi$) if  the above inequality remains true for $A\otimes id_{\mathcal{M}}$, where $\mathcal{M}$ is any finite von Neumann algebra. The case $f(x)=x\ln(x)$ has been studied in a series of paper, see \cite{LJR, bgj20, LJL}.  By imposing more conditions on $f$, we would recover most properties  of $\clsi$s such as  stability under tensorization and change of measure.  
\hhz

Intriguingly, the study of generalized monotone metrics in the space of quantum states sheds light on $\cfsi$s and the Bregman relative entropy.  The monotone metric was anticipated by Morozova and Chenstov (\cite{mc}) to transfer the geometric techniques to the noncommutative settings. Motivated by  Morozova and Chenstov, Petz (\cite{Petz96}) introduced monotone metrics  systematically using the relative modular operators and discovered the equivalent relation between operator monotone functions and the monotone metrics. Later on, Hiai and Petz  (\cite{HP12}) extended the monotone metrics to two parameters. Continuing Petz' study, we define define the generalized monotone metrics associated to two-variable functions via the double operator integral. By this new definition of generalized monotone metrics, we explore a wide range of Sobolev type inequalities.
\hhz

The paper is organized as follows. In section 2, we introduce the generalized monotone metrics. In section 3, we define $\cfsi$s and establish $\cfsi$s for derivation triples. In section 4, we discuss examples and applications such as complete Beckner inequalities and random transpositions and Bernoulli-Laplace models.

\section{Generalized Monotone Metrics} 
\subsection{Monotone metrics}
Let $\mathcal{N}$ be a finite von Neumann algebra equipped with a normal faithful tracial state $\tau$ and $\beta: \mathcal{N}\to\mathcal{N}$ be a completely positive trace preserving (CPTP) map. The set of positive elements in $\mathcal{N}$ is denoted by $\mathcal{N}_{+}$. 
Let $L_{p}(\mathcal{N},\tau)$ denote the noncommutative $L_{p}$ space, written as $L_{p}(\cN)$ if the trace $\tau$ is clear from the context.  Let $\Bbb{R}^{+}=(0,\infty)$ in  the sequel.  The left and right multiplications by $\rho\in \mathcal{N}$ are defined by
\begin{equation*} \ml_{\rho}(a)=\rho a\quad \text{   and    } \quad \mr_{\rho}(a)=a\rho,\quad  \forall a\in\mathcal{N}.
\end{equation*}
Note $\ml_{\rho}$ and $\mr_{\sigma}$ commute for any $\rho,\sigma\in\mathcal{N}$.
For $\rho,\sigma\in\mathcal{N}_{+}$ and $f:\Bbb{R}^{+}\to \Bbb{R}^{+}$,  we define $\Bbb{J}_{\rho,\sigma}^{f}: \mathcal{N}\to\mathcal{N}$ by 
\begin{equation}\Bbb{J}_{\rho,\sigma}^{f}=f(\ml_{\rho}\mr_{\sigma}^{-1}) \mr_{\sigma} \label{J1},
\end{equation}
where $\ml_{\rho}R_{\sigma}^{-1}$ is the \textit{relative modular operator},  see \cite{petz07} for more information. We use   $\Bbb{J}_{\rho}^{f}$  if $\rho=\sigma$. The inverse of $\Bbb{J}_{\rho,\sigma}^{f}$ is given by 
$$\left(\Bbb{J}_{\rho,\sigma}^{f}\right)^{-1}=f^{-1}(\ml_{\rho}\mr_{\sigma}^{-1})\mr_{\sigma}^{-1}.$$ Let  $\rho,\sigma\in\mathcal{N}_{+}$ and $f: \Bbb{R}^{+}\to\Bbb{R}^{+}$, then the following conditions are equivalent (\cite{HP12}): 
\begin{align}
\beta^{*} (\Bbb{J}_{\beta(\rho),\beta(\sigma)}^{f})^{-1}\beta &\leq (\Bbb{J}_{\rho,\sigma}^{f})^{-1}; \label{equiv-1}\\
\beta \Bbb{J}_{\rho,\sigma}^{f}\beta^{*} &\leq \Bbb{J}_{\beta(\rho),\beta(\sigma)}^{f}. \label{equiv-2}
\end{align}
Let us recall the following generalized Lieb's concavity theorem (\cite{Petz85, HP12, HP}).
\begin{theorem} \label{HPconvex} Let $\beta:\mathcal{N}\to\mathcal{N}$ be a CPTP map and $f:\Bbb{R}^{+}\to\Bbb{R}^{+}$ be an operator monotone function. Assume that $\rho,\sigma\in\mathcal{N}_{+}$, then 
\begin{align*} \beta^{*} \left(\Bbb{J}_{\beta(\rho),\beta(\sigma)}^{f} \right)^{-1} \beta\leq \left(\Bbb{J}_{\rho,\sigma}^{f}\right)^{-1}.
\end{align*}
\end{theorem}
\noindent Hiai and Petz usually require that $\rho,\sigma, \beta(\rho),\beta(\sigma)$ are invertible. As we pointed out in \cite{LJL} that it is enough to assume the positivity by perturbation argument $\rho+\epsilon I$ for $\epsilon\to 0^{+}.$   Consequently Hiai and Petz defined the \textit{monotone metrics with two parameters} $\gamma_{\rho,\sigma}^{f}$ by 
\begin{align}\gamma_{\rho,\sigma}^{f}(a,b)=\langle a, \left(\Bbb{J}_{\rho,\sigma}^{f}\right)^{-1}(b) \rangle, \quad \forall a,b\in\mathcal{N},\label{mm}
\end{align}
where  $\langle a,b \rangle=\tau(a^{*}b)$ is the Hilbert-Schmidt inner product. For $\rho,\sigma\in\mathcal{N}_{+}$ and an operator monotone function $f:\Bbb{R}^{+}\to \Bbb{R}^{+}$, we have 
$$\gamma_{\beta(\rho),\beta(\sigma)}^{f}(\beta(a),\beta(a))\leq \gamma_{\rho,\sigma}^{f}(a,a), \quad a\in\mathcal{N}.$$
\begin{corollary}\label{mjc}
For an operator monotone function $f$, the monotone metric $\gamma^{f}_{\rho,\sigma}(a,a)$ is a jointly convex function for $(\rho,\sigma,a)$ for $\rho,\sigma\in\mathcal{N}_{+}$ and $a\in\mathcal{N}$.
\end{corollary}
 
\subsection{Generalized monotone metrics} 
Let us recall that for $F:\Bbb{R}^{+}\times \Bbb{R}^{+}\to\Bbb{R}^{+}$ and $\rho,\sigma\in\mathcal{N}_{+}$ the double operator integral is defined by
$$Q^{\rho,\sigma}_{F}(a)=\int_{0}^{\infty}\int_{0}^{\infty} F(s,t) dE_{\rho}(s) a dE_{\sigma}(t),$$
where $E_{\rho}((s,t])=1_{(s,t]}(\rho)$ is the spectral projection of $\rho$. We denote it by $Q_{F}^{\rho}$ if $\rho=\sigma$. For a comprehensive account of the double operator integral, see \cite{krein1,krein2,bs1,bs2,bs3,PS10}. For operators $\rho=\sum_{i=1}^{k}s_{i}p_{i}$ and $\sigma=\sum_{j=1}^{l}t_{j}q_{j}$ with discrete specrtum, this simplifies to a Schur multiplier
$$Q^{\rho,\sigma}_{F}(y)=\sum_{i=1}^{k}\sum_{j=1}^{l} F(s_{i},t_{j})p_{i}yq_{j}, \quad  \forall y\in \mathcal{N}.$$ 
Note that
$\left(Q_{F}^{\rho,\sigma}\right)^{-1}=Q_{F^{-1}}^{\rho,\sigma}.$
Let $f_{[0]}(x,y)=f(\frac{x}{y})y$ for $f: \Bbb{R}^{+}\to\Bbb{R}^{+}$, then $Q_{f_{[0]}}^{\rho,\sigma}=\Bbb{J}^{f}_{\rho,\sigma}.$
Let us introduce two families of functions:
\begin{align}
\mathfrak{C}^{-}&=\{F;\: \beta Q^{\rho,\sigma}_{F} \beta^{*}\leq Q^{\beta(\rho),\beta(\sigma)}_{F}, \forall \rho,\sigma\in\mathcal{N}_{+} \text{ and CPTP } \beta \}, \label{c-}\\
\mathfrak{C}^{+}&=\{F;\: \beta^{*}Q^{\beta(\rho), \beta(\sigma)}_{F}\beta\leq Q^{\rho,\sigma}_{F},\forall \rho,\sigma\in\mathcal{N}_{+} \text{ and CPTP } \beta \} . \label{c+}
\end{align}
\begin{definition} Let $F\in\mathfrak{C}^{+}$ and $\rho, \sigma\in\mathcal{N}_{+}$. We define the {\rm{(}}two-variable{\rm{)}} \textit{generalized monotone metric} $\gamma_{\rho,\sigma}^{F}:\mathcal{N}\to\mathcal{N}$ by
\begin{align}\gamma_{\rho,\sigma}^{F}(a,b)=\langle a, Q_{F}^{\rho,\sigma}(b) \rangle.\label{gmm}
\end{align}
\end{definition}
\noindent It follows from the definition that 
\begin{align}\label{tgmm}\gamma_{\beta(\rho),\beta(\sigma)}^{F}(\beta(a),\beta(a))\leq \gamma_{\rho,\sigma}^{F}(a,a),\quad \forall a\in\mathcal{N}.\end{align}
We use the same notation as \eqref{mm} defined by \cite{HP12},  but we only refer to \eqref{mm} if the superscript function $f$ is one-variable.  Let $f$ be operator monotone, then we identify $$\gamma_{\rho,\sigma}^{f}=\gamma^{F}_{\rho,\sigma}$$ with $F=f^{-1}_{[0]}.$ \begin{theorem} \label{tgmc} Let $F\in\mathfrak{C}^{+}$ satisfying $\lambda F(\lambda x,\lambda y)\leq F(x,y)$
for any $\lambda\in[0,1]$. Then the generalized monotone metric $\gamma_{\rho,\sigma}^{F}(a,a)$ is a convex function for $(\rho,\sigma,a)$ of $\rho,\sigma\in\mathcal{N}_{+}$ and $a\in\mathcal{N}$.
\end{theorem}
\begin{proof} We use the standard trick and consider $\beta: \Mz_{2}\otimes\mathcal{N}\to\mathcal{N}$ defined by 
\begin{align*}
\left(\begin{matrix}  x_{1}&x_{2}\\x_{3}&x_{4}
\end{matrix}\right)\mapsto x_{1}+x_{4}.
\end{align*}
Then $\beta$ is CPTP. Let $\rho=\left(\begin{smallmatrix} \lambda \rho_{1} &0\\0&(1-\lambda)\rho_{2} \end{smallmatrix}\right)$, $\sigma=\left(\begin{smallmatrix} \lambda \sigma_{1} &0\\0&(1-\lambda)\sigma_{2} \end{smallmatrix}\right)$, and $a=\left(\begin{smallmatrix} \lambda a_{1}&0\\0&(1-\lambda)a_{2} \end{smallmatrix}\right)$ for some $\lambda\in[0,1]$.  By \eqref{tgmm}, we obtain that 
\begin{align*}
&\gamma^{F}_{\lambda\rho_{1}+(1-\lambda)\rho_{2},\lambda\sigma_{1}+(1-\lambda)\sigma_{2}}(\lambda a_{1}+(1-\lambda)a_{2},\lambda a_{1}+(1-\lambda)a_{2})\\
\leq & \gamma^{F}_{\lambda\rho_{1},\lambda \sigma_{1}}(\lambda a_{1},\lambda a_{1})+ \gamma^{F}_{(1-\lambda)\rho_{2},(1-\lambda) \sigma_{2}}((1-\lambda) a_{2},(1-\lambda) a_{2}).
\end{align*}
We further have 
\begin{align}\label{ppp}
\gamma_{\lambda \rho, \lambda \sigma}^{F}(\lambda a, \lambda a)\leq \lambda \gamma^{F}_{\rho,\sigma}(a,a).
\end{align}
Indeed 
\begin{align*}
\gamma_{\lambda \rho, \lambda \sigma}^{F}(\lambda a, \lambda a)=&\lambda^{2}\langle a, \int_{0}^{\infty} \int_{0}^{\infty} F(x,y) dE_{\lambda \rho}(x)a dE_{\lambda \sigma(y)}\rangle\\
=&\lambda^{2}\langle a,\int_{0}^{\infty}\int_{0}^{\infty}F(\lambda x, \lambda y)dE_{\rho}(x)adE_{\sigma}(y) \rangle \\
\leq &\lambda \gamma^{F}_{\rho,\sigma}(a,a).
\end{align*}
Applying \eqref{ppp} completes the proof. 
\end{proof}
\noindent The monotonicity of $\gamma^{F}_{\rho,\sigma}(a,a)$ does not necessarily imply the joint convexity for $F\in\mathfrak{C}^{+}$ since the condition $\lambda F(\lambda x,\lambda y)\leq F(x,y)$ sometimes fails. Let $f$ be operator monotone and $F=f_{[0]}^{-1}$, we actually have the equality.

\begin{prop} \label{dcp} We have the following properties.
\begin{itemize}
\item[\rm{(1)}] The  sets $\mathfrak{C}^{+}$ and $\mathfrak{C}^{-}$ are positive cones.
\item[\rm{(2)}] Let $F_{1}\in\mathfrak{C}^{+}$ and $F_{2}\in\mathfrak{C}^{-}$, let $F'_{1}(x,y)=F_{1}(x+t, y+s)$ and $F'_{2}(x,y)=F_{2}(x+t,y+s)$ for any fixed $t,s\geq 0$. Then $F'_{1}\in\mathfrak{C}^{+} $ and $F'_{2}\in\mathfrak{C}^{-}.$
\item[\rm{(3)}] If $F\in\mathfrak{C}^{+}$, then $\frac{1}{F}\in\mathfrak{C}^{-}$. Similarly if $F\in\mathfrak{C}^{-}$, then $\frac{1}{F}\in\mathfrak{C}^{+}$.
\item[\rm{(4)}] Let $f: \Bbb{R}^{+}\to\Bbb{R}^{+}$ be operator monotone, then $f_{[0]}\in\mathfrak{C}^{-}$ and $f_{[0]}^{-1}\in\mathfrak{C}^{+}$.
\end{itemize}
\end{prop}
\begin{proof} We only give proofs for (3) and (4). The equivalence between
\begin{align}
\beta^{*} (Q^{\beta(\rho),\beta(\sigma)}_{F})^{-1}\beta \leq (Q^{\rho,\sigma}_{F})^{-1} \label{equi-1}
\end{align}
and 
\begin{align}
\beta Q^{\rho,\sigma}_{F}\beta^{*} \leq Q^{\beta(\rho),\beta(\sigma)}_{F}. \label{equi-2}
\end{align}
yields (3). (4) follows directly from Theorem \ref{HPconvex} and (3).
\end{proof}
\begin{example}\label{hpl} Let $f(x)=\frac{x-1}{\ln(x)}$, then
$f$ is operator monotone. Indeed, 
$f(x)=\int_{0}^{1}x^{r}dr$
and $x^{r}$ is operator monotone for $r\in[0,1]$. 
Then we have  $$f_{[0]}(x,y)=\frac{x-y}{\ln(x)-\ln(y)}\in\mathfrak{C}^{-} \quad \text{and} \quad f_{[0]}^{-1}(x,y)=\frac{\ln(x)-\ln(y)}{x-y}\in\mathfrak{C}^{+}.$$
\end{example}

\noindent Let $f^{[1]}(x,y)=\frac{f(x)-f(y)}{x-y}$ denote the difference quotient of $f$. We consider the following set 
\begin{align} 
\mathfrak{c}^{+}=\{f'|f^{[1]}\in\mathfrak{C}^{+}\}.
\end{align}

\begin{prop} \label{cp} We have the following properties.
\item[\rm{(1)}]  The set $\mathfrak{c}^{+}$ is a positive cone.
\item[\rm{(2)}] The set $\mathfrak{c}^{+}$ is invariant under right translation.
\item[\rm{(3)}]  Let $f(x)=\frac{1}{x+k}$ with $k\geq 0$, then $f\in\mathfrak{c}^{+}$. 
\end{prop}
\begin{proof} We only give the proof of (3). Let $g(x)=\ln(x+k)$. By Example \ref{hpl} and Proposition \ref{dcp}, we have $g^{[1]}\in\mathfrak{C}^{+}$. If follows from the definition that $f=g'\in\mathfrak{c}^{+}$.
\end{proof}
\begin{example} Let $F(x,y)=\frac{x^{p}-y^{p}}{x-y}$ for $p\in(0,1)$.
Let us recall that $$f(x)=x^{p}=\frac{sin(p\pi)x}{\pi}\int_{0}^{\infty}\frac{r^{p-1}}{r+x}dr.$$
By Proposition \ref{cp}, we have $f'\in \mathfrak{c}^{+}$ and $F=f^{[1]}\in\mathfrak{C}^{+}$.
We cannot find a function $f$ such that $F^{-1}=f_{[0]}$. Thus $\mathfrak{C}^{+}$ is a strict extension of the operator monotone functions.
\end{example}
\begin{remark} In a discussion, we noticed that Haonan Zhang also gave a proof of the example above separately. Haonan Zhang was trying to develop the matrix-valued Beckner inequalities using the geodesic convexity techniques in \cite{CM2} and \cite{CM}.
\end{remark}
\section{Complete Sobolev type inequality} 
\subsection{Derivations}
Let $\mathcal{N}$ be a finite von Neumann algebra equipped with a normal faithful tracial state $\tau$.
Let $_{\mathcal{N}}\Hs_{\mathcal{N}}$ be a self-adjoint Hilbert $\mathcal{N}$-$\mathcal{N}$ bimodule with the antilinear form $J$. A  derivation of a von Neumann algebra $\mathcal{N}$ is a densely defined linear operator $\delta: L_{2}(\mathcal{N},\tau)\rightarrow \Hs$ such that
\begin{itemize}[leftmargin=6.0mm,nolistsep]
\item[(1)] $\dom(\delta)$ is a weakly dense $^*$-subalgebra in $\mathcal{N}$;
\item[(2)] the identity element $1\in \dom(\delta)$;
\item[(3)] $\delta(xy)=x\delta(y)+\delta(x)y$, for any $x,y\in \dom(\delta)$.
\end{itemize}
\noindent 
We always work with a closable derivation and denote the closure by $\bar{\delta}$.
A derivation $\delta$ is said to be \textit{$*$-preserving} if $J(\delta(x))=\delta(x^*)$. Every closable $*$-preserving derivation $\delta$ determines a positive operator $\delta^{*}\bar{\delta}$ on $L_{2}(\mathcal{N},\tau)$. 
It was shown  in \cite{Sau} that $T_{t}=e^{-t\delta^{*}\bar{\delta}}: \mathcal{N} \rightarrow \mathcal{N}$ is a strongly continuous semigroup of CPTP maps. 
 See\cite{SA}, \cite{AS}, \cite{Jesse}, \cite{Kap}, and \cite{BR} for more details.
The functional calculus of a derivation $\delta$ is given by
\begin{align}\label{fcd} \delta(f(\rho))=Q_{f^{[1]}}^{\rho}(\delta(\rho))=\int_{0}^{\infty} \int_{0}^{\infty} \frac{f(s)-f(t)}{s-t}dE_{\rho}(s)\delta(\rho)dE_{\rho}(t).
\end{align}

Now let $T_t=e^{-tA}:\mathcal{N}\rightarrow \mathcal{N}$ be a strongly continuous semigroup of completely positive unital self-adjoint maps on $L_2(\mathcal{N},\tau)$. The generator $A$ is a positive operator on $L_{2}(\mathcal{N},\tau)$ given by $$A(x)=\lim_{t\rightarrow 0^+} \frac{1}{t}(T_t(x)-x),\forall x\in \dom(A).$$
It was pointed out  in \cite{Sau} that  $\dom(\delta)=\{x\in \mathcal{N}| \|A^{1/2}x\|_2<\infty\}$ is indeed a $*$-algebra and  invariant under the semigroup.
The weak gradient form of $A$ is defined by
$$\Gamma_A(x,y)(z)=\frac{1}{2}(\tau(A(x)^*yz)+\tau(x^*A(y)z)-\tau(x^*yA(z))).$$
If the weak gradient form $\Gamma_A(x,y)\in L_{1}(\mathcal{N})$ for all $x,y\in \dom(A^{1/2})$, we say the generator $A$ (or $
T_{t}$) satisfies \textit{$\Gamma$-regularity}.
\begin{theorem}[\cite{JRS18}] \label{jsr} If $A$ satisfies $\Gamma$-regularity, then there exists a finite von Neumann algebra $(\mathcal{M},\tau)$ containing $\mathcal{N}$ and a $*$-preserving derivation $\delta_{A}:\dom(A^{1/2})\rightarrow L_2(\mathcal{M})$ such that
\begin{equation}\label{inducedder}\tau(\Gamma_{A}(x,y)z)=\tau(\delta_{A}(x)^*\delta_{A}(y)z).\end{equation}
Equivalently $\Gamma_{A}(x,y)=E_{\mathcal{N}}(\delta_{A}(x)^*\delta_{A}(y))$, where $E_{\mathcal{N}}: \mathcal{M}\rightarrow \mathcal{N}$ is the conditional expectation.
\end{theorem}
Throughout the paper,  we always work with a closable $*$-preserving derivation $\delta$ and a strongly continuous semigroup $T_{t}=e^{-tA}$ of completely positive unital self-adjoint maps on $L_2(\mathcal{N},\tau)$  satisfying $\Gamma$-regularity.   

\subsection{Generalized Fisher information}
The $f$-Fisher information $I^{f,\tau}_{A}$ of $A$ is defined as
$$I^{f,\tau}_{A}(\rho)=\tau(A(\rho) f'(\rho)), \forall \rho\in\dom(A^{1/2})\cap L_{2}(\mathcal{N}) \text{ and } f'(\rho)\in L_{\infty}(\mathcal{N}).$$
\noindent Equivalently  $$I^{f,\tau}_{A}(\rho)=\lim_{\epsilon\rightarrow 0^+}\tau(A(\rho)f'(\rho+\epsilon 1)).$$ 
For a derivation $\delta$, the Fisher information is defined as
\begin{align}\label{fisher-d}
I_{\delta}^{f,\tau}(\rho)=\tau\left(\delta(\rho)Q^{\rho}_{f^{[2]}}(\delta(\rho))\right), \quad \forall\rho\in \dom(\delta)\subset \mathcal{N},
\end{align}
where 
$$f^{[2]}(x,y)=\frac{f'(x)-f'(y)}{x-y}.$$
Then $I_{\delta}^{f}(\rho)=I_{\delta^{*}\bar{\delta}}^{f}(\rho).$ We use $I_{A}^{f}$ or $I_{\delta}^{f}$ if the trace is clear from the context.  In the rest of this section, we always consider  convex and continuously differentiable $f: \Bbb{R}^{+}\to \Bbb{R}^{+}$ such that $f^{[2]}\in\mathfrak{C}^{+}$.
By Theorem \ref{jsr}, for any $A$ satisfying $\Gamma$-regularity, there exists a closable $*$-preserving derivation $\delta_{A}:\dom(A^{1/2})\to L_{2}(\mathcal{M})$ such that $\Gamma_{A}(x,y)=E_{\mathcal{N}}(\delta_{A}(x)^*\delta_{A}(y))$ where $E_{\mathcal{N}}: \mathcal{M}\to\mathcal{N}$.
Thus $$I_{A}^{f}(\rho)=I_{\delta_{A}}^{f}(\rho).$$ The choice of $\delta_{A}$ is not necessarily unique, but $I_{A}^{f}$ is uniquely determined. We recapture the widely used  Fisher information $I_{A}(\rho)=\tau(A(\rho)\ln(\rho))$ by choosing $f(x)=x\ln(x)$.  We shall also observe the relation between the $f$-Fisher information and the generalized monotone metric 
$$I_{\delta}^{f}(\rho)=\gamma^{f^{[2]}}_{\rho,\rho}(\delta(\rho),\delta(\rho)).$$ 
\begin{lemma}[non-negativity] \label{finn}The $f$-Fisher information is nonnegative.
\end{lemma}
\begin{proof}  The convexity of $f$ implies that $f^{[2]}\geq 0$.  Set $w=(Q^{\rho}_{f^{[2]}})^{1/2}(\delta_{A}(\rho)) $ with $$(Q^{\rho}_{f^{[2]}})^{1/2}(y)=\int_{0}^{\infty}\int_{0}^{\infty} \left(\frac{f'(s)-f'(t)}{s-t}\right)^{1/2}dE_{\rho}(s)y dE_{\rho}(t).$$ Thus
\begin{align}I_{A}^{f}(\rho)&=\tau(\delta_{A}(\rho)Q_{f^{[2]}}^{\rho}(\delta_{A}(\rho)))\\
&=\tau\left(E_\mathcal{N}(ww) \right)\geq 0.
\end{align}
Similarly $I_{\delta}$ is also nonnegative.
\end{proof}
\noindent An important example is $f(x)=x^{p}$ for $p\in(1,2)$, and we denote such $p$-Fisher information by $I_{A}^{p}$ or $I_{\delta}^{p}$. As an application of Theorem \ref{tgmc}, we get the following result.
\begin{corollary}\label{pfc} The $p$-Fisher information is convex.
\end{corollary}
Recall that for any finite von Neumann algebra $\mathcal{N}$,
there exists a $\sigma$-finite measure space $(X,\mu)$ such that $\mathcal{Z}(\mathcal{N})\cong L_{\infty}(X,\mu)$ and $\mathcal{N}=\int_{X} \mathcal{N}_{x}d\mu(x)$, where $\mathcal{Z}(\mathcal{N})$ is the center of $\mathcal{N}$ and $\mathcal{N}_{x}$ is a factor for any $x\in X$.  Now we rewrite the $f$-Fisher information  by using the direct integral
  $$I^{f,\tau}_{\delta}(\rho)=\int_{X} I^{f,\tau_{x}}_{\delta}(\rho_{x})d\mu(x).$$
\begin{lemma} \label{fip} Let $\tau_{1}$ and $\tau_{2}$ be normal faithful traces over $\mathcal{N}$ and $\frac{d\tau_{1}}{d\tau_{2}}\geq c$ for some  $c>0$. Then for any $\rho\in\mathcal{N}_{+}$,
$$c I_{A}^{f,\tau_{2}}(\rho)\leq I_{A}^{f,\tau_{1}}(\rho).$$
The result remains true for $I_{\delta}^{f}$.
\end{lemma}
\begin{proof} Two traces only differ by two measures $\mu_{1}$ and $\mu_{2}$ over the center $L_{\infty}(X,\mu_{1})\cong L_{\infty}(X,\mu_{2})\cong Z(\mathcal{N})$ .  Note that $\frac{d\tau_{1}}{d\tau_{2}}\geq c$ if and only $\frac{d\mu_{1}}{d\mu_{2}}\geq c$. Setting the pointwise differential form $w_{x}=(Q^{\rho_{x}}_{f^{[2]}})^{1/2}(\delta_{A}(\rho_{x}))$, we infer that 
$$cI_{A}^{f,\tau_{2}}(\rho)=c\int_{X} \tau_{2}\left(E_{\mathcal{N}}( w_{x}w_{x})\right) d\mu_{2}(x)\leq \int_{X}\tau_{1}\left(E_{\mathcal{N}}( w_{x}w_{x})\right) \mu_{1}(x)=I_{A}^{f,\tau_{1}}(\rho).$$
\end{proof}

\subsection{ Bregman relative entropy}
Let us recall the  definition of $f$-Bregman relative entropy
\[
  d^{f,\tau}(\rho\|\sigma) =    \tau(f(\rho)-f(\sigma)-(\rho-\sigma)f'(\sigma)),
\]
for $\rho,\sigma\in\mathcal{N}_{+}$.  For simplicity, we would call it as $f$-relative entropy. Equivalently $$d^{f}(\rho\|\sigma)=\lim_{\epsilon\rightarrow 0^{+}}d^{f}(\rho\|\sigma+\epsilon 1).$$ 
We write $d^{f}(\rho\|\sigma)$ if the trace $\tau$ is clear from the context. For a comprehensive study of Bregman relative entropy, see \cite{mpv16, pv15, vir16}.
 It follows from the definition that $d^{f}(\rho\|\sigma)\geq 0$ with the equality if and only if $\rho=\sigma$.
Note that we identify the Lindblad relative entropy $$d^{f}(\rho\|\sigma)=\tau(\rho\ln \rho-\rho\ln \sigma-\rho+\sigma)=D_{\lind}(\rho\|\sigma)$$  with the choice $f(x)=x\ln(x)$. The $f$-relative entropy admits an integral representation (\cite{pv15})
\begin{align*}
d^{f}(\rho\|\sigma)=\int_{s=0}^{1}\tau\left((\rho-\sigma) \frac{d}{dt}f(\sigma+(s+t)(\rho-\sigma))|_{t=0}\right) ds.
\end{align*}
Let $\mathcal{K}\subset \mathcal{N}$ be a von Neumann subalgebra of $\mathcal{N}$ and $E_{\mathcal{K}}$ be the conditional expectation onto $\mathcal{K}$. The relative entropy with respect to $\mathcal{K}$ is defined by 
$$d^{f}_{\mathcal{K}}(\rho)=d^{f}(\rho\|E_{\mathcal{K}}(\rho)).$$
Noting $\tau\left ((\rho-E_{\mathcal{K}}(\rho))f'(E_{\mathcal{K}}(\rho)) \right)=0$, then 
\begin{align}d^{f}_{\mathcal{K}}(\rho)=\tau(f(\rho)-f(E_{\mathcal{K}}(\rho))).\end{align}
\begin{lemma}\label{nme}
The following equality remains true 
$$d^{f}(\rho\|\sigma)=d^{f}_{\mathcal{K}}(\rho)+d^{f}(E_{\mathcal{K}}(\rho)\|\sigma)$$
for any $\sigma\in\mathcal{K}.$ 
\end{lemma}
\begin{proof} Note $\tau(\rho f'(\sigma))
=\tau\left(E_{\mathcal{K}}(\rho)f'(\sigma)\right)$. Then we have 
\begin{align*}
d^{f}(\rho\|\sigma)&=\tr\left(f(\rho)-f(E_{\mathcal{K}}(\rho))+f(E_{\mathcal{K}}(\rho))-f(\sigma)+(\rho-\sigma)f'(\sigma)\right)\\
&=d^{f}_{\mathcal{K}}(\rho)+\tau\left(f(E_{\mathcal{K}}(\rho))-f(\sigma)+(E_{\mathcal{K}}(\rho)-\sigma)f'(\sigma)\right)\\
&=d_{\mathcal{K}}^{f}(\rho)+d^{f}(E_{\mathcal{K}}(\rho)\|\sigma).
\end{align*}
\end{proof}
\noindent Together with nonnegativity of $f$-relative entropy, Lemma \ref{nme} implies that 
\begin{align}\label{dcs}d^{f}_{\mathcal{K}}(\rho)=\inf_{\sigma\in \mathcal{K}}d^{f}(\rho\|\sigma).
\end{align}
Let $\mathcal{N}_{\fix}\subset \mathcal{N}$ be the fixed point algebra of the semigroup $T_{t}=e^{-tA}$ and $E$ be the conditional expectation onto $\mathcal{N}_{\fix}$,  then 
$ET_{t}=T_{t}E=E.$ 
\begin{lemma}[gradient form] \label{gf}  The semigroup $T_{t}$ relates the $f$-relative entropy and the $f$-Fisher information .The $f$-Fisher information is the negative derivative of $d_{\mathcal{N}_{\fix}}^{f}(T_{t}(\rho))$,
$$\frac{d}{dt}d^{f}_{\mathcal{N}_{\fix}}(T_{t}(\rho))=-I_{A}^{f}(T_{t}(\rho)).$$
\end{lemma}
\begin{proof} Let $g(t)=d^{f}_{\mathcal{N}_{\fix}}(\rho_{t})=\tau(f(\rho_{t})-f(E(\rho_{t})))$. By the chain rule, we obtain that 
\begin{align*}
g'(t)=\frac{d}{dt} \tau(f(\rho_{t})-f(E(\rho))=\tau\left(-A(T_{t}\small(\rho\small))) f'(T_{t}(\rho))\right)=-I_{A}^{f}(T_{t}(\rho)).
\end{align*}
\end{proof}

\begin{lemma} \label{fdp} Let $\tau_{1}$ and $\tau_{2}$ be two normal faithful traces over a finite von Neumann algebra $\mathcal{N}$ such that $\frac{d\tau_{1}}{d\tau_{2}}\leq c$ for $c>0$. For any $\rho,\sigma \in \mathcal{N}_{+}$,
$$d^{f,\tau_{1}}(\rho\|\sigma)\leq cd^{f,\tau_{2}}(\rho\|\sigma).$$
In particular, we have $d^{f,\tau_{1}}_{\mathcal{K}}(\rho) \leq c d^{f,\tau_{2}}_{\mathcal{K}}(\rho).$
\end{lemma}
\begin{proof} We follow the notations and idea in the proof of Lemma \ref{fip}.
Also note that $\frac{d\tau_{1}}{d\tau_{2}}\leq c$ if and only if $\frac{d\mu_{1}}{d\mu_{2}}\leq c$. Again by the non-negativity of the $f$-relative entropy, we have
\begin{align*}
d^{f,\tau_{1}}(\rho\|\sigma)= &\int_{X} d^{f,\tau_{1}}(\rho_{x}\|\sigma_{x})d\mu_{1}(x) \\
\leq &c \int_{X} d^{f,\tau_{2}}(\rho_{x}\|\sigma_{x})d\mu_{2}(x) = cd^{f,\tau_{2}}(\rho\|\sigma).
\end{align*}
The second assertion follows from \eqref{dcs}.
\end{proof}

\begin{theorem}[Data Processing Inequality]\label{dpi} Let $\Phi: \mathcal{N}\to\mathcal{N}$ a quantum channel (CPTP), then
$$d^{f}(\Phi(\rho)\|\Phi(\sigma))\leq d^{f}(\rho\|\sigma)\quad \forall \rho\in\mathcal{N}_{+}, \sigma\in\mathcal{N}_{\fix}.$$
\end{theorem}
\begin{proof}
Let $a(t)=(1-t)\sigma+t\rho$ for $t\in[0,1]$, and we consider the function 
$$H_{\rho,\sigma}(t)=\tau(f(a(t))).$$
By chain rule $H'(t)=\tau(f'(a(t))(\rho-\sigma))$. It follows from the integration by parts that
\begin{align} \label{di}
\int_{0}^{1}(1-t)H''_{\rho,\sigma}(t)dt=&(1-t) H''_{\rho,\sigma}(t)|_{0}^{1}-\int_{0}^{1} (-1)H'_{\rho,\sigma}(t)dt\\
=&-H'_{\rho,\sigma}(0)+H_{\rho,\sigma}(1)-H_{\rho,\sigma}(0)=d^{f}(\rho\|\sigma).
\end{align}
Recall that $\lim_{t\to 0^{+}}\frac{g(\rho+t\sigma)-g(\rho)}{t}=Q_{g^{[1]}}^{\rho}(\sigma)$, then we have 
\begin{align*}
H''_{\rho,\sigma}(t)=&\tau\left(\lim_{\epsilon\to 0^{+}} \frac{f'(a(t+\epsilon))-f'(a(t)) }{\epsilon} (\rho-\sigma) \right)\\
=&\tau\left(\lim_{\epsilon \to 0^{+}} \frac{f'(a(t)+\epsilon(\rho-\sigma))-f'(a(t))}{\epsilon} (\rho-\sigma) \right)\\
=& \tau\left((\rho-\sigma) Q^{a(t)}_{f^{[2]}} (\rho-\sigma) \right)=\gamma_{a(t),a(t)}^{f^{[2]}}(\rho-\sigma,\rho-\sigma).
\end{align*}
Then $f^{[2]}\in\mathfrak{C}^{+}$ implies that 
$$H''_{\Phi(\rho),\Phi(\sigma)}(t)\leq H''_{\rho,\sigma}(t).$$
Together with \eqref{di} it yields the assertion.
\end{proof}
\noindent Let $f(x)=x^{p}$ for $p\in(1,2)$. We obtain the $p$-relative entropy
\begin{align}\label{dp}d^{p}(\rho\|\sigma)=\tau(\rho^{p}-\sigma^{p}-p(\rho-\sigma)\sigma^{p-1}).
\end{align}
Thus\begin{align}\label{dpc}d^{p}_{\mathcal{K}}(\rho)=\tau(\rho^{p}-(E_{\mathcal{K}}(\rho))^{p}).\end{align}

It shall be noted that $p$-relative entropy is different from the (sandwiched) R\'enyi entropy.  The $p$-relative entropy with respect the conditional expectation \eqref{dpc} appeared in \cite{bt06}, where they studied the classical (commutative) situations and ergodic systems. The general properties of Bregman relative entropy are systematically studied in 
\cite{pv15, vir16, mpv16}.
 \noindent Again we use the standard argument as in Theorem \ref{tgmc} and obtain the joint convexity. 
\begin{corollary} 
The $f$-relative entropy $d^{p}(\rho\|\sigma)$ is a jointly convex function for $(\rho,\sigma)$ for $\rho,\sigma\in\mathcal{N}_{+}$ if $d^{f}(\lambda \rho\|\lambda \sigma)\leq \lambda d^{f}(\rho\|\sigma)$ for any $\lambda \in[0,1]$. Thus $D_{\lind}$ and $d^{p}$ are jointly convex.
\end{corollary}

\subsection{$\cfsi$}
\begin{definition} The semigroup $T_{t}=e^{-tA}$ or the generator $A$ with the fixed-point algebra $\mathcal{N}_{\fix}$ is said to satisfy:
\begin{itemize}[leftmargin=6mm]
\item[{\rm{(1)}}] the modified $f$-Sobolev inequality $\lambda$-$\mfsi$ {\rm{(}}with respect to the trace $\tau${\rm{)}} if there exists a constant $\lambda>0$ such that
$$\lambda d_{\mathcal{N}_{\fix}}^{f}(\rho)\leq  I_{A}^{f}(\rho),\quad \forall \rho\in \dom(\delta)\cap \mathcal{N}_{+};$$
{\rm{(}}or equivalently $d_{\mathcal{N}_{\fix}}^{f}(T_{t}(\rho))\leq e^{-\lambda t} d^{f}_{\mathcal{N}_{\fix}}(\rho),\: \forall \rho\in \mathcal{N}_{+}.${\rm{)}}
\item[{\rm{(2)}}] the complete $f$-Sobolev inequality $\lambda$-$\mfsi$ {\rm{(}}with respect to the trace $\tau${\rm{)}} if $A \otimes id_{\mathcal{F}}$ satisfies $\lambda$-$\mfsi$ for any finite von Neumann algebra $\mathcal{F}$.
\end{itemize}
Let $\cfsi(A,\tau)$ be the supremum of $\lambda$ such that $A$ satisfies $\lambda$-$\cfsi$, or denoted by $\cfsi(A)$ if there is no ambiguity.  Sometimes we use $\cfsi(T_{t})$ for convenience. 
\end{definition}
\noindent $\cfsi$ is a generalization of the complete logarithmic Sobolev inequality. An important example is $f(x)=x^{p}$ for $p\in(1,2)$, which induces the so-called $\cpsi$ and $\clsi^{+}$
\cite{LJL}.
\begin{lemma}  \label{ddi} Let $E_{\mathcal{K}}:\mathcal{N}\to \mathcal{K}$ be a conditional expectation. Then 
\begin{align*} 
I_{I-E_{\mathcal{K}}}^{f}(\rho) =d^{f}(\rho\| E_{\mathcal{K}}(\rho))+d^{f}(E_{\mathcal{K}}(\rho)\|\rho)
\end{align*}
and hence $\cfsi(I-E_{\mathcal{K}})\geq 1.$ 
\end{lemma}
\begin{proof}
We have
\begin{align*}
d^{f}(E_{\mathcal{K}}(\rho)\| \rho)=&\tau \left(  f(E_{\mathcal{K}}(\rho)) -f(\rho)-(E_{\mathcal{K}}(\rho)-\rho)f'(\rho) \right)\\
=&\tau\left(f(E_{\mathcal{K}}(\rho)-f(\rho))\right)+\tau\left( (I-E_{\mathcal{K}})(\rho)f'(\rho) \right)\\
=&-d^{f}(\rho\|E_{\mathcal{K}}(\rho))+I_{I-E_{\mathcal{K}}}^{f}(\rho).
\end{align*}
Thus $\cfsi(I-E_{\mathcal{K}})$ follows from the nonnegativity of $f$-relative entropy.
\end{proof}
\noindent This result  for $\clsi$ was given by \cite{DPR17}. We can obtain a better constant for $\cpsi$, see \cite{LJL}. 
Applying Lemma \ref{gf}, we have the following equivalence:
\begin{prop} \label{equfi}
The following conditions are equivalent:
\begin{enumerate}
\item[{(\rm{1)}}] $\lambda d^{f}_{\mathcal{N}_{\fix}}{\rho}\leq I_{A}^{f}(\rho)$ for any $\rho\in\mathcal{N}_{+}$;
\item[{\rm{(2)}}] $d_{\mathcal{N}_{\fix}}^{f}(T_{t}(\rho)) \leq e^{-\lambda t}d_{\mathcal{N}_{\fix}}^{f}(\rho)$ for any $\rho\in\mathcal{N}_{+}$.
\end{enumerate}
\end{prop}

\noindent Combining Lemma \ref{fdp} and Lemma \ref{fip}, we obtain the following \textit{change of measure principle}.
\begin{theorem}[change of measure principle]  \label{sutp} Let $\tau_{1}$ and $\tau_{2}$ be normal faithful traces over $\mathcal{N}$ and $c_{2}\leq \frac{d\tau_{1}}{d\tau_{2}}\leq c_{1}$ for some $c_{1},c_{2}>0$. Then $\cfsi(A,\tau_{1})\geq \frac{c_{2}}{c_{1}}\cfsi(A,\tau_{2})$.
\end{theorem}
\noindent The change of measure principle is often referred to as Holley and Stroock (\cite{hs87}) argument. They proved that LSIs are stable under change of measures. This remains true for CLSIs \cite{LJL}.
\noindent $\cfsi$s are also stable under tensorization as an application of the  data processing inequality.
\begin{theorem} [tensorization stability] 
Let $T^{j}_{t}: \mathcal{N}_{j}\to \mathcal{N}_{j}$ be a family of semigroups with
fixed-point algebras $\mathcal{N}_{\fix,j}\subset \mathcal{N}_{j}$ for $1\leq j\leq k$ . Then the tensor semigroup $T_{t}=\otimes_{j=1}^{k}T_{t}^j$ has the fixed-point algebra $\mathcal{N}_{\fix}=\otimes_{j=1}^{k}\mathcal{N}_{\fix,j}$. Moreover, we have
$$\cfsi(T_{t})\geq \inf_{1\leq j\leq k} \cfsi(T^{j}_{t}).$$
\end{theorem}
\begin{proof} It suffices to prove for  the 2-fold tensor product. For the $n$-fold tensor product, we may use the standard induction argument. 
Let $E$, $E_{1}$, and $E_{2}$ be the conditional expectations onto $\mathcal{N}_{\fix}$, $\mathcal{N}_{\fix,1}$ and $\mathcal{N}_{\fix,2}$ respectively. Applying \eqref{nme} and Theorem \ref{dpi} gives
\begin{align*} 
d^{f}_{\mathcal{N}_{\fix}}(\rho)=&d^{f}(\rho\| E_{1}\otimes \id_{2} (\rho))+d^{f}(E_{1}\otimes \id_{2} (\rho)\| E(\rho))\\
\leq &d^{f}(\rho\| E_{1}\otimes \id_{2} (\rho))+d^{f}(\rho\| \id_{1}\otimes E_{2}(\rho)).
\end{align*}
For the $n$-fold tensor product, we may use the standard induction argument.
\end{proof}
\noindent The following result is motivated by \cite{Spo78} and \cite{LJR} (lemma 2.6).

\begin{theorem}\label{iexp} Let $T_{t}=e^{-tA}$ be a semigroup of completely positive self-adjoint unital maps on $L_{2}(\mathcal{N},\tau)$. Suppose there exists some positive constant $\lambda$ such that 
$$I_{A}^{f}(T_{t}(\rho))\leq e^{-\lambda t}I_{A}^{f}(\rho), \forall \rho\in\mathcal{N}_{+},$$
then $\cfsi(A)\geq \lambda.$
\end{theorem}
\begin{proof}  We use Lemma \ref{gf} again. Define $g(t)=d^{f}_{\mathcal{N}_{\fix}}(T_{t}(\rho))$, then $$-g'(t)\leq -e^{-\lambda t} g'(0).$$
Integrating both sides from $[0, \infty)$ yields $\cfsi(A)\geq \lambda$.
\end{proof}
\begin{remark} In many situations we do not need the condition $f^{[2]}\in\mathfrak{C}^{+}$, such as  Lemma \ref{finn},  \ref{fip},\ref{nme}, \ref{gf}, \ref{fdp}, \ref{ddi}, Prop \ref{equfi}, Theorem \ref{sutp} and \ref{iexp}. However, this condition is necessary to obtain the date processing inequality.
\end{remark}

 \subsection{$\cfsi$ of derivation triple}
Let us recall the definition of a derivation triple in \cite{LJL}. Let $\mathcal{N}$ be a finite von Neumann algebra equipped with a normal faithful tracial state $\tau$, and $\delta$ be a closable $*$-preserving  derivation on $\mathcal{N}$. Suppose there exists a larger finite von Neumann algebra $(\mathcal{M},\tau)$ containing $\mathcal{N}$ and a weakly dense $^*$-subalgebra $\mathcal{A}\subset \mathcal{N}$ such that
\begin{enumerate}[leftmargin=6mm]
\item[(1)] $\mathcal{A}\subset \dom(\delta);$
\item[(2)]$\delta^{*}\bar{\delta}: \mathcal{A}\to\mathcal{A};$ 
\item[(3)] $\delta: \mathcal{A}\rightarrow L_{2}(\mathcal{M},\tau)$.
\end{enumerate}
We define $\pi_{\delta}: \Omega^{1}(\mathcal{A})\rightarrow \mathcal{M}$ by
$$\pi_{\delta}(a\otimes b-1\otimes ab)=\delta(a)b,$$
where
$ \Omega^{1}(\mathcal{A})=\{\sum_{j}(a_{j}\otimes b_{j}-1\otimes a_{j}b_{j})|a_{j}, b_{j}\otimes \mathcal{A}\}\subset \mathcal{A}\otimes \mathcal{A}. $ Thus $\Omega_{\delta}(\mathcal{A})$ is Hilbert $\mathcal{A}$-bimodule with inner product
$$(\delta(a_{1})b_{1}, \delta(a_{2})b_{2})_{\mathcal{A}}=b_{1}^{*}E_{\mathcal{N}}(\delta(a_{1}^{*})\delta(a_{2}))b_{2},$$
where $E_{\mathcal{N}}: \mathcal{M}\rightarrow \mathcal{N}$ is the conditional expectation and $(\cdot, \cdot)_{\mathcal{A}}$ is the $\mathcal{N}$-valued inner product.  A linear operator $\Rc: \Omega_{\delta}(\mathcal{A})\to \mathcal{M}$ is called the Ricci operator of $\triple$ provided that
\begin{enumerate}[leftmargin=6mm]
 \item[(1)]  $\Rc(a\rho b)=a\Rc(\rho) b,\quad \forall a,b\in\mathcal{A},\rho\in\Omega_{\delta}(\mathcal{A})$;
 \item[(2)] there  exists a strongly continuous semigroup $\hat{T}_{t}=e^{-tL}: \mathcal{M}\rightarrow \mathcal{M}$ of completely positive trace preserving maps such that $\Upgamma_{L}(a,b)=E_{\mathcal{N}}(\delta(a^{*})\delta(b))$  and 
$\delta (\delta^{*}\bar{\delta} a)-L(\delta(a))=\Rc(\delta(a))$
for any $ a,b\in\mathcal{A}$.
\end{enumerate}
The derivation $\delta$ is said to admit a Ricci curvature $\Rc\geq\lambda$ bounded below by a constant $\lambda$,  if $( \Rc(\rho),\rho )_{\mathcal{A}}\geq \lambda E_{\mathcal{N}}(\rho^{*}\rho)$ for any $\rho\in \Omega_{\delta}(\mathcal{A})$. We say the generator  $A$ of $T_{t}=e^{-tA}$ admits $\Rc\geq\lambda$ if there exists a derivation triple  $\triple$ such that  
$$\Gamma_{A}(a,b)=E_{\mathcal{N}}(\delta(a^{*})\delta(b)), \quad\forall a,b\in\mathcal{A}$$
and $\delta$ admits $\Rc\geq\lambda.$  It shall be noted that the choice of $\delta$ is not unique, thus we may find a larger Ricci lower bound of  $A$ by choosing a good $\delta$.

\begin{lemma} Let $\triple$ be a derivation triple with a Ricci curvature $\Rc\geq \lambda>0$. Then 
$$ I_{\delta}^{f}(e^{-t\delta^{*}\bar{\delta}}(\rho))\leq e^{-2\lambda t} I_{\delta}^{f}(\rho),\quad \forall\rho\in\mathcal{N}_{+}.$$
\end{lemma}
\begin{proof} 
In the proof, we use the following notations $A=\delta^{*}\bar{\delta}$, $T_{t}=e^{-tL}$, and $T_{t}(\rho)=\rho_{t}$. Let $\hat{T}_{t}=e^{-tL}$ be the semigroup given in the definition of the Ricci curvature. Let us consider two functions
\begin{align*}
h(t)=I_{\delta}^{f}(\rho_{t})=\tau\left(\int_{0}^{\infty}\int_{0}^{\infty}\delta(\rho_{t})\frac{f'(s)-f'(l)}{s-l}dE_{\rho_{t}}(s)\delta(\rho_{t})dE_{\rho_{t}}(l) \right)
\end{align*}
and 
\begin{align*}
k(t)=&\gamma^{f^{[2]}}_{\rho_{t},\rho_{t}}(\hat{T}_{t}(\delta(\rho)),\hat{T}_{t}(\delta(\rho)) )\\
=&\tau\left(\int_{0}^{\infty}\int_{0}^{\infty} \hat{T}_{t}(\delta(\rho))\frac{f'(s)-f'(l)}{s-l}  dE_{\rho_{t}}(s) \hat{T}_{t}(\delta(\rho))dE_{\rho_{t}}(l)\right).
\end{align*}
By $\Upgamma_{L}(a,b)=E_{\mathcal{N}}(\delta(a^{*})\delta(b))$, then 
$$k(t)=\gamma^{f^{[2]}}_{\hat{T}_{t}(\rho),\hat{T}_{t}(\rho)}(\hat{T}_{t}(\delta(\rho)),\hat{T}_{t}(\delta(\rho)) ).$$
Noting $f^{[2]}\in\mathfrak{C}^{+}$, we deduce that $k(t)\leq k(0)$ and   
\begin{align} \label{kn}
k'(0)\leq 0.
\end{align}
By the product rule, we have
\begin{align*}
h'(t)=&-2 \tau (\int_{0}^{\infty}\int_{0}^{\infty}\delta(A(\rho_{t}))\frac{f'(s)-f'(l)}{s-l}dE_{\rho_{t}}(s)\delta(\rho_{t})dE_{\rho_{t}}(l)))+h'_{r}(t)
\end{align*}
and 
\begin{align*}
k'(t)=&-2\tau(\int_{0}^{\infty}\int_{0}^{\infty} L(\hat{T}_{t}(\delta(\rho)))\frac{f'(s)-f'(l)}{s-l}dE_{\rho_{t}}(s) \hat{T}_{t}(\delta(\rho)) dE_{\rho_{t}}(l)))+k'_{r}(t),
\end{align*}
where $h'_{r}$ and $k'_{r}$ are the derivatives corresponding to $dE_{\rho_{t}}$. We make an important observation 
\begin{align} \label{hki}
 h'_{r}(0)=k'_{r}(0).
\end{align}
Together with $\delta (\delta^{*}\bar{\delta} a)-L(\delta(a))=\Rc(\delta(a))$, then  
\begin{align*}h'(0)-k'(0)=&-2\tau(\int_{0}^{\infty}\int_{0}^{\infty}\Rc(\delta(\rho)))\frac{f'(s)-f'(l)}{s-l}dE_{\rho}(s)\delta(\rho)dE_{\rho}(l)))\\
=&-2\tau\left(\Rc\left(Q^{\rho}_{f^{[2]}})^{1/2}(\delta(\rho))\right) (Q^{\rho}_{f^{[2]}})^{1/2}(\delta(\rho))  \right).
\end{align*}
Using that $\Rc\geq\lambda$, we infer that
\begin{align*} h'(0)-k'(0)\leq -2\lambda h(0).\end{align*}
By \eqref{kn}, then \begin{align} \label{ii}h'(0)\leq -2\lambda h(0).\end{align}
Setting
$h_{s}(t)=I_{\delta}^{f}(\rho_{t+s})$, then $h'_{s}(0)=h'(s)$. Inequality \eqref{ii} remains true for $h_{s}$. Hence
$$h'(s)=h_{s}'(0)\leq -2\lambda h_{s}(0)=-2\lambda h(s)$$
 completes the proof.
\end{proof}
\noindent Applying Theorem \ref{iexp}, we get the following complete Sobolev inequality.
\begin{theorem} \label{main}Let $\triple$ be a derivation triple  with a Ricci curvature $\Rc\geq \lambda>0$. Let $f:\Bbb{R}^{+}\to \Bbb{R}^{+}$ be continuously differentiable and $f^{[2]}\in\mathfrak{C}^{+}$.  Then we have $$\cfsi\triple \geq 2\lambda.$$
\end{theorem}

\section{Applications}
\subsection{$p$ norm estimate}
\begin{theorem} Let $\triple$ be a derivation triple with a Ricci curvature $\Rc\geq \lambda>0$ and $T_{t}=e^{-t\delta^{*}\bar{\delta}}$. Then we have 
$$ \|T_{t}(\rho)-E(\rho)\|_{p}\leq e^{-\lambda t}\sqrt{\frac{2}{p(p-1)}} \|\rho\|_{p}^{1-p/2}(\|\rho\|_{p}^{p}-\|E(\rho)\|_{p}^{p})^{1/2},\quad \forall \rho\in\mathcal{N}_{+}.$$
\end{theorem}
\begin{proof} This proof is inspired by \cite{RX}.  For self-adjoint $a,b\in\mathcal{N}$, we define 
$$G_{a,b}(s)=\|a+sb\|_{p}^{p}-\frac{p(p-1)}{2}s^{2}\|a+sb\|_{p}^{p-2} \|b\|_{p}^{2}.$$
It follows from the definition that $G_{a,b}$ is convex over $\Bbb{R}$ if $G_{x,y}''(0)\geq 0$ for any self-adjoint $x,y\in\mathcal{N}$. In \cite{RX}, they considered the following function 
$$\psi(s)=\|a+sb\|_{p}^{p}$$
with an additional condition that $a$ is invertible. They proved that 
\begin{align}
\psi''(0)\geq p(p-1)\|a\|_{p}^{p-2}\|b\|_{p}^{2}.
\end{align}
It implies that $G_{a,b}''(0)\geq 0$. This remains true if $a$ is not invertible, see the proof of Theorem 2 in \cite{RX}. Let $a=E(\rho_{t})=E(\rho)$ and $b=\rho_{t}-E(\rho)$, then 
\[\|a+sb\|_{p}^{p}\geq \| E(a+sb)\|_{p}^{p}=\|a\|_{p}^{p}.\]
Hence the right derivative of $G_{a,b}$ at $0$ is nonnegative,
and convexity implies  $G_{a,b}'(s)\geq 0$ for any $s\geq 0$. In particular $G(1)\geq G(0)$, then 
$$\|E(\rho)\|_{p}^{p}+\frac{p(p-1)}{2}\|\rho_{t}\|_{p}^{p-2}\|\rho_{t}-E(\rho)\|_{p}^{2}\leq \|\rho_{t}\|_{p}^{p}.$$
By $\cpsi$, we have 
\begin{align*}
\|\rho_{t}\|_{p}^{p}\leq \|E(\rho)\|_{p}^{p}+e^{-2\lambda t} \left(\|\rho\|_{p}^{p} -\|E(\rho)\|_{p}^{p} \right).
\end{align*}
Chaining the two inequalities gives 
\begin{align*}
\|\rho_{t}-E(\rho)\|_{p}^{2} \leq  e^{-2\lambda t} \frac{2}{p(p-1)} \|\rho_{t}\|_{p}^{2-p}\left(\|\rho\|_{p}^{p} -\|E(\rho)\|_{p}^{p} \right).
\end{align*}
Noting $\|\rho_{t}\|_{p}\leq \|\rho\|_{p}$ and  taking the square root of the inequality above complete the proof.
\end{proof}

\subsection{Bakry-\'Emery criterion}
Let $(M,g)$ be a smooth $n$-dimensional Riemannian manifold without boundary. For a smooth function $U\in C^{\infty}(M)$, we define a probability measure $\mu$ by
$$d\mu=\frac{1}{\Zu}e^{-U}dvol$$ with the normalization factor $\Zu=\int_{M} e^{-U}dvol$ and a Bakry-\'Emery Ricci curvature $\Rcu$ by $$\Rcu=\Rc+\Hess(U).$$ Applying  Theorem \ref{main} , we obtain the modified Laplace operator 
$$\Du=\Delta+\nabla U \cdot \nabla ,$$
where $\Delta$ is the Laplace-Beltrami operator.  See \cite{LJL} for the definition of a derivation triple of a Riemannian manifold.
\begin{theorem} Let $(M, g, \mu)$ be a smooth Riemannian manifold with the measure $\mu$ defined by $d\mu=\frac{1}{\Zu} e^{-U} dvol$ with $\Zu=\int_{M} e^{-U}dvol$ for $U\in C^{\infty}(M)$. Given that  $\Rcu\geq \kappa>0$ and $f^{[2]}\in\mathfrak{C}^{+}$, then
$$\cfsi(\Du)\geq 2\kappa.$$
\end{theorem}
\noindent Let $f(x)=x^{p}$, we obtain the complete Beckner inequalities.
\begin{corollary}[complete Beckner Inequalities] 
Let $(M, g, \mu)$ be a smooth Riemannian manifold with the measure $\mu$ defined by $d\mu=\frac{1}{\Zu} e^{-U} dvol$ with $\Zu=\int_{M} e^{-U}dvol$ for $U\in C^{\infty}(M)$. Given that  $\Rcu\geq \kappa>0$, then
$$\cpsi(\Du)\geq 2\kappa.$$
\end{corollary}
\noindent By Theorem \ref{sutp}, we have the following result.
\begin{theorem} 
Let $\nu$ be the probability measure defined by $d\nu=\frac{1}{\Zv}e^{-V}dvol$ with $V\in C^{\infty}(M)$, where $\Zv$ is the normalization factor.  If $\|U-V\|_{\infty}\leq C$ and $f^{[2]}\in\mathfrak{C}^{+}$, then
$$e^{2C}\cfsi(\Du)\geq  \cfsi(\Dv)
 .$$
\end{theorem}

\subsection{Random Transpositions}
Let $S_{n}$ be the permutation group on $\{1,\dots, n\}$, and we consider the Laplace operator $\Delta_{n}$ given by
$$(\Delta_{n}f)(\sigma)=\frac{1}{n}\sum_{i,j=1}^{n} \left[f(\sigma)-f(\sigma^{ij})\right],$$
where $\sigma^{ij}$ denotes the configuration of $\sigma$ after swapping the elements on $i$-th site and $j$-th site. For example let $\sigma=(1\quad 3\quad 2)$ and $i=1$ and $j=2$, then $\sigma^{ij}=(3\quad 1\quad 2)$. It is well-known that $\Delta_{n}$ is ergodic, and thus $E_{n}(f)=\frac{1}{n!}\sum_{\sigma\in S_{n}}f(\sigma).$ 
The lower bound of Ricci curvature $\Rc$ of the random transposition on the Symmetric group $S_{n}$ is defined using the geodesic convexity, and $\Rc\geq \frac{4}{n}$ (\cite{emt} and \cite{fm}). MLSIs were studied in \cite{goel}, \cite{bt06}, and \cite{gao} using the martingale methods in \cite{ly98}, and they proved that
\begin{align*}
1\leq \mlsi(\Delta_{n}) \leq 4.
\end{align*}
The upper was given by the spectral gap $\lambda_{2}(\Delta_{n})=2$ (\cite{ds87}).
We also apply the martingale methods and establish a similar relation between $\cpsi(\Delta_{n+1})$ and $\cpsi(\Delta_{n})$: 
\begin{theorem} \label{rt} Let $p\in(1,2)$, then $$p\leq \cpsi(\Delta_{n})\leq 4\quad \text{and}\quad 1\leq \clsi(\Delta_{n})\leq 4$$ for any $n\geq 2$. 
\end{theorem}
\noindent As pointed out by \cite{bt06} (Section 4), the upper bound of $\mpsi(\Delta)$ is also given by the spectral gap. It is sufficient to give the lower bound.  Let $\mathcal{M}$ be a finite von Neumann algebra equipped with a normal faithful trace $\tau$, and we consider $\mathcal{M}$-valued function. For $f\in\ell_{\infty}^{m}$, let $\tau(f)=\frac{1}{m}\sum_{j=1}^{m}\tau(f(j))$.
\begin{lemma}\label{rtl} For any $\mathcal{M}$-valued function $f$ defined over $\{1,\dots, n\}$ and $n\geq 2$, we have
$$\tau(f^{p}-E(f))\leq \frac{1}{2n^2}\sum_{i,j=1}^{n}\tau\left[ \left(f(i)-f(j)\right)\left({f(i)}^{p-1}-{f(j)}^{p-1}\right)\right],$$
where $E(f)=\frac{1}{n}\sum_{i=1}^{n}f(i)$.
\end{lemma}
\noindent This lemma is an immediate application of $\cpsi(I-E)\geq p$ (\cite{LJL}).
\noindent The scalar case of the following lemma was proven in \cite{bt06}, and here we give an operator valued version. 
\begin{lemma} \label{rtc} Let $p\in(1,2)$  and $F(\rho,\sigma)=\tau\left[(\rho-\sigma)(\rho^{p-1}-\sigma^{p-1})\right]$, then $F$ is jointly convex for $x,y\in\mathcal{M}_{+}$.
\end{lemma}
\begin{proof} Let $f=(\rho,\sigma)$ and  $E(f)=\frac{1}{2}(\rho+\sigma)$, then $\delta(f)=\frac{1}{4}(f(2)-f(1), f(1)-f(2))$ and $\delta^{*}\delta=I-E$. We can rewrite $F$ as  $F(\rho,\sigma)=\frac{8}{p}I_{\delta}^{p}(f)$. By Corollary \ref{pfc} $F$ is jointly convex.
\end{proof}
\noindent Now we prove Theorem \ref{rt}.
\begin{proof}
Let $\mathcal{N}_{i}\subset L_{\infty}(S_{n+1},\mathcal{M})$ be a von Neumann subalgebra generated by $\{e_{i}^{j}\}_{j=1}^{n+1}$ satisfying 
$$e_{i}^{j}(\sigma)=\begin{cases}&1,\quad \text{if}\quad \sigma_{i}=j;\\&0, \quad \text{otherwise}.\end{cases}$$
We denote the corresponding conditional expectation by $E_{\mathcal{N}_{i}}$. By martingale equality, we have 
\begin{align*}
d^{p}(f\| E(f))=d^{p}(f\| E_{\mathcal{N}_{i}}(f))+d^{p}( E_{\mathcal{N}_{i}}(f)\| E(f)  ).
\end{align*}
Since $i$ is also uniformly chosen from from the $n+1$ sites, then 
\begin{align}
d^{p}(f\| E(f))=\frac{1}{n+1}\sum_{i=1}^{n+1}\Big[d^{p}(f\| E_{\mathcal{N}_{i}}(f))+d^{p}( E_{\mathcal{N}_{i}}(f)\| E(f)  )\Big].
\end{align}
For any fixed $i$, we define
\begin{align*}
f_{i}(j)&=\sum_{\sigma_{i}=j}f(\sigma).
\end{align*}
Let $$E_{ij}(f)(\sigma)=\frac{1}{n!}\begin{cases}
f_{i}(j),&\quad \text{if}\quad \sigma_{i}=j;  \\ 0,&\quad\text{otherwise},\end{cases} $$
then $$E_{\mathcal{N}_{i}}(f)=\frac{1}{n+1}\sum_{j=1}^{n+1}E_{ij}(f).$$ 
Let us define a projection map $$P_{ij}(f)(\sigma)=\begin{cases}f(\sigma), &\text{if}\quad \sigma_{i}=j; \\0, &\text{otherwise}, \end{cases}$$ then
\begin{align*}
d^{p}(f\|E_{\mathcal{N}_{i}}(f))=\frac{1}{(n+1)!}\sum_{j=1}^{n+1} \tau\left(P_{ij}(f)^{p}-E_{ij}(f)^{p}\right)
\end{align*}
Also by our assumption \begin{align*}&\frac{\cpsi(\Delta_{n})}{(n+1)!}  \tau\left(P_{ij}(f)^{p}-E_{ij}(f)^{p}\right) \\&\leq \frac{p}{(n+1)!2n}\sum_{\{(\sigma,k,l)| \sigma_{i}=\sigma^{kl}_{i}=j\}}\tau\left[ \left(f(\sigma^{kl})-f(\sigma)\right)\left( f(\sigma^{kl})^{p-1}-f(\sigma)^{p-1} \right) \right].
\end{align*}
It is important to observe that 
\begin{align*}
&\sum_{i=1}^{n+1}\sum_{j=1}^{n+1}\sum_{\{(\sigma,k,l)| \sigma_{i}=\sigma^{kl}_{i}=j\}}\tau\left[ \left(f(\sigma^{kl})-f(\sigma)\right)\left( f(\sigma^{kl})^{p-1}-f(\sigma)^{p-1} \right) \right]\\
=&(n-1)\sum_{\sigma,k,l}\tau\left[ \left(f(\sigma^{kl})-f(\sigma)\right)\left( f(\sigma^{kl})^{p-1}-f(\sigma)^{p-1} \right) \right].
\end{align*}
Thus
\begin{align}\label{rti1} \frac{1}{n+1}\sum_{i=1}^{n+1}d^{p}(f\| E_{\mathcal{N}_{i}}(f))\leq \frac{n-1}{n \cpsi(\Delta_{n})} I_{\Delta_{n+1}}^{p}(f).
\end{align}
Applying Lemma \ref{rtl}, we obtain that
$$d^{p}( E_{\mathcal{N}_{i}}(f)\| E(f))\leq\frac{1}{2(n+1)^{2}}\sum_{k,l=1}^{n+1}\tau\Big[ \left(E_{ik}(f)-E_{il}(f)\right)\left(E_{ij}(f)^{p-1}-E_{il}(f)^{p-1}\right)\Big].$$
By the definition of $E_{ik}$, we have 
\begin{align*}
&\left(E_{ik}(f)-E_{il}(f)\right)\left(E_{ij}(f)^{p-1}-E_{il}(f)^{p-1}\right)\\ =
&\left(\frac{1}{n!}\sum_{\sigma_{i}=l}f(\sigma^{kl})-\frac{1}{n!}\sum_{\sigma_{i}=l}f(\sigma)\right) \left(\left(\frac{1}{n!}\sum_{\sigma_{i}=l}f(\sigma^{kl})\right)^{p-1}-\left(\frac{1}{n!}\sum_{\sigma_{i}=l}f(\sigma)\right)^{p-1}\right).
\end{align*}
 Together with Lemma \ref{rtc}, it implies that
 \begin{align*}
 &\tau\left[\left(E_{ik}(f)-E_{il}(f)\right)\left(E_{ij}(f)^{p-1}-E_{il}(f)^{p-1}\right)\right]\\ \leq &\frac{1}{n!}\sum_{\sigma_{i}=l}\tau\left[(f(\sigma^{kl})-f(\sigma))(f(\sigma^{kl})^{p-1}-f(\sigma)^{p-1}) \right].
 \end{align*}
 Then we have \begin{align*} &d^{p}( E_{\mathcal{N}_{i}}(f)\| E(f))\\\leq& \frac{1}{2(n+1)^{2}n!}\sum_{k,l=1}^{n+1}\sum_{\sigma_{i}=l}\tau\left[(f(\sigma^{kl})-f(\sigma))\left(f(\sigma^{kl})^{p-1}-f(\sigma)^{p-1}\right) \right].\end{align*}Thus
 \begin{align}\label{rti2}\frac{1}{n+1}\sum_{i=1}^{n+1}d^{p}( E_{\mathcal{N}_{i}}(f)\| E(f))\leq \frac{1}{(n+1)p}I_{\Delta_{n+1}}^{p}(f). 
 \end{align}
 Combining \eqref{rti1} and \eqref{rti2}, we obtain 
$$\frac{1}{\cpsi(\Delta_{n+1})}\leq \frac{n-1}{n\cpsi(\Delta_{n})}+\frac{1}{p(n+1)} .$$
Lemma \ref{rtl} implies that 
$\cpsi(\Delta_{2})\geq 2p.$ By induction method, we have
$$\cpsi(\Delta_{n+1})\geq p.$$
Indeed, by assuming that $\cpsi(\Delta_{n})\geq p$ we obtain that  $$\frac{1}{\cpsi(\Delta_{n+1})}\leq \frac{1}{p}\left(\frac{n-1}{n}+\frac{1}{n+1}\right)\leq \frac{1}{p}.$$ 
We only prove the estimate for $\cpsi$ and the argument remains true for $\clsi$.
\end{proof} 
\subsection{Bernoulli-Laplace Model}
We consider the Bernoulli-Laplace model with $n$ distinct sites $\{1,\dots, n\}$ and $r$ identical particles, where $n\geq 2$ and $1\leq r\leq n-1$. Each site can be occupied by at most $1$ particle. Let $C_{n,r}$ be the state space of the configurations of $r$ elements occupying $n$ sites. The Laplace operator $\Delta_{n,r}: L_{\infty}(C_{n,r})\to L_{\infty}$ is defined by 
\begin{align*}
(\Delta_{n,r}f)(\sigma)=\frac{1}{n}\sum_{i< j} \left[f(\sigma)-f(\sigma^{ij}) \right].
\end{align*}
Again $\sigma^{ij}$ is the configuration of $\sigma$ after we swap the $i$-th site and the $j$-th site. Let $\sigma_{i}$ denote the number of particles occupying the $i$-th site. The lower bound of the Ricci curvature of the BL model was also studied in \cite{emt} and \cite{fm}. \cite{goel}, \cite{bt06}, and \cite{gao} proved that
\begin{align*}
1/2\leq \mlsi(\Delta_{n,r}) \leq 2
\end{align*}
Again we use the noncommutative martingale method and obtain a similar estimate.
\begin{theorem} \label{bl} Let $p\in(1,2)$, then $$p/2\leq \cpsi(\Delta_{n,r})\leq 2 \quad \text{and} \quad 1/2\leq \clsi(\Delta_{n,r})\leq 2$$ for any $n\geq 2$ and $1\leq r\leq n-1$.
\end{theorem} 
\noindent Again the upper bound of $\mpsi(\Delta)$ is also given by the spectral gap, see \cite{ds87}.  Let $\mathcal{M}$ be a finite von Neumann algebra equipped with a normal faithful trace $\tau$, and we consider $\mathcal{M}$-valued function. For $f\in\ell_{\infty}^{m}$, let $\tau(f)=\frac{1}{m}\sum_{j=1}^{m}\tau(f(j))$.
\noindent The proof of \ref{bl} is quite similar to the proof of \ref{rt}.  
\begin{proof} Let $\mathcal{N}_{i}\subset L_{\infty}(C_{n+1,r},\mathcal{M})$ be a von Neumann subalgebra generated by $\{e_{i}^{0},e_{i}^{1}\}$ defined by
\begin{align*}
e_{i}^{j}(\sigma)=\begin{cases} 1,&\quad\text{if} \quad \sigma_{i}=j;\\0,&\quad\text{otherwise},\end{cases} \quad j=0,1.
\end{align*}
Let $E_{\mathcal{N}_{i}}$ be the corresponding conditional expectation onto $\mathcal{N}_{i}$. 
By martingale equality, we have that 
\begin{align*}
d^{p}(f\| E(f))=d^{p}(f\| E_{\mathcal{N}_{i}}(f))+d^{p}( E_{\mathcal{N}_{i}}(f)\| E(f)  ).
\end{align*}
Since $i$ is uniformly chosen from from the $n+1$ sites, then 
\begin{align}
d^{p}(f\| E(f))=\frac{1}{n+1}\sum_{i=1}^{n+1}\Big[d^{p}(f\| E_{\mathcal{N}_{i}}(f))+d^{p}( E_{\mathcal{N}_{i}}(f)\| E(f)  )\Big].
\end{align}
For any fixed $i$, we define 
\begin{align*}
f_{i}(j)=\sum_{\sigma_{i}=j} f(\sigma), \quad j=0,1.
\end{align*}
Let 
\begin{align*}
E_{i,j}(f)(\sigma)=\frac{1}{a_{j}} \begin{cases} f_{i}(j), &\quad \text{if} \quad \sigma_{i}=j ;\\ 0,&\quad \text{otherwise}, \end{cases}
\end{align*} where  $a_{0}=\ {n-1\choose r}$ and $a_{1}={n-1 \choose r-1}.$
Let us define projections $$P_{i,j}(f)(\sigma)=\begin{cases}f(\sigma), &\text{if}\quad \sigma_{i}=j; \\0, &\text{otherwise}, \end{cases}\quad  j=0,1.$$ Then
\begin{align*}
d^{p}(f\|E_{\mathcal{N}_{i}}(f))= \frac{1}{{n+1\choose r}} \sum_{j=0,1} \tau\left[ P_{i,j}(f)^{p}-E_{i,j}(f)^{p}\right].
\end{align*}
By the definition of $\cpsi$, then we have 
\begin{align*}
&\cpsi(\Delta_{n,r-j})\tau\left[ P_{i,j}(f)^{p}-E_{i,j}(f)^{p}\right]\\\leq& \frac{p}{2n} \sum_{\{(\sigma,k,l)|\sigma_{i}=\sigma_{i}^{kl}=j\}}\tau\left[  \left(f(\sigma^{kl})-f(\sigma)\right)\left( f(\sigma^{kl})^{p-1}-f(\sigma)^{p-1} \right) \right].
\end{align*}
An important observation is that
\begin{align*}
&\sum_{i=1}^{n+1}\sum_{j=0,1}\sum_{\{(\sigma,k,l)| \sigma_{i}=\sigma^{kl}_{i}=j\}}\tau\left[ \left(f(\sigma^{kl})-f(\sigma)\right)\left( f(\sigma^{kl})^{p-1}-f(\sigma)^{p-1} \right) \right]\\
=&(n-1)\sum_{\sigma,k,l}\tau\left[ \left(f(\sigma^{kl})-f(\sigma)\right)\left( f(\sigma^{kl})^{p-1}-f(\sigma)^{p-1} \right) \right].
\end{align*}
Thus \begin{align}\label{bl1}\frac{1}{n+1}\sum_{i=1}^{n+1}d^{p}(f\| E_{\mathcal{N}_{i}}(f))\leq \frac{n-1}{n} \frac{1}{\min_{j=0,1}\{\cpsi(\Delta_{n,r-j})\}}I_{\Delta_{n+1,r}}(f).
\end{align}
Applying Lemma \ref{rtl}, we obtain that
\begin{align*}&d^{p}( E_{\mathcal{N}_{i}}(f)\| E(f))\\\leq&\frac{a_{0}a_{1}}{(a_{0}+a_{1})^{2}}\tau\Big[ \left(\frac{f_{i}(1)}{a_{1}}-\frac{f_{i}(0)}{a_{0}}\right)\left( \left(\frac{f_{i}(1)}{a_{1}}\right)^{p-1}-\left(\frac{f_{i}(0)}{a_{0}}\right)^{p-1}\right)\Big].\end{align*}
The definition of $f_{i}(j)$
 infers that
\begin{align*}
f_{i}(1)=&\sum_{\sigma_{i}=1}f(\sigma)=\frac{1}{(n-r+1)}\sum_{k=1}^{n+1}\sum_{\sigma_{k}=1,\sigma_{i}=0}f(\sigma^{ki}),\\
f_{i}(0)=&\sum_{\sigma_{i}=0}f(\sigma)=\frac{1}{r}\sum_{k=1}^{n+1}\sum_{\sigma_{k}=1,\sigma_{i}=0}f(\sigma).
\end{align*}
Together with Lemma \ref{rtc}, it implies that
 \begin{align*}
 &\tau\left[\left(E_{i,1}(f)-E_{i,0}(f)\right)\left(E_{i,1}(f)^{p-1}-E_{i,0}(f)^{p-1}\right)\right]\\ 
 \leq &\frac{(n-r)!(r-1)!}{n!}\sum_{k=1}^{n+1}\sum_{\sigma_{i}=0,\sigma_{k}=1}\tau\Big[\left(f(\sigma^{ik})-f(\sigma)\right)\left(f(\sigma^{ik})^{p-1}-f(\sigma)^{p-1}\right)\Big].
 \end{align*}
 Similarly
 \begin{align*}
  &\tau\left[\left(E_{i,1}(f)-E_{i,0}(f)\right)\left(E_{i,1}(f)^{p-1}-E_{i,0}(f)^{p-1}\right)\right]\\ 
 \leq &\frac{(n-r)!(r-1)!}{n!}\sum_{k=1}^{n+1}\sum_{\sigma_{i}=1,\sigma_{k}=0}\tau\Big[\left(f(\sigma^{ik})-f(\sigma)\right)\left(f(\sigma^{ik})^{p-1}-f(\sigma)^{p-1}\right)\Big].
 \end{align*}
Thus 
\begin{align*}&d^{p}( E_{\mathcal{N}_{i}}(f)\| E(f))\\\leq& \frac{a_{0}a_{1}}{2(a_{0}+a_{1})^{2}} \frac{(n-r)!(r-1)!}{n!}\sum_{k=1}^{n+1}\sum_{\sigma_{i}\neq \sigma_{k}}\tau\Big[\left(f(\sigma^{ik})-f(\sigma)\right)\left(f(\sigma^{ik})^{p-1}-f(\sigma)^{p-1}\right)\Big].
\end{align*}
 Then we have 
  \begin{align*}
\sum_{i=1}^{n+1} d^{p}( E_{\mathcal{N}_{i}}(f)\| E(f)) \leq\frac{2}{p} I_{\Delta_{n+1,r}}^{p}(f).
 \end{align*}
  Together with \eqref{bl1}, it implies that 
 \begin{align*}
 \frac{1}{\cpsi{\Delta_{n+1,r}}}\leq \frac{n-1}{n}\frac{1}{\min_{j=0,1}\{\cpsi(\Delta_{n,r-j})\}}+\frac{2}{(n+1)p}.
 \end{align*}
 Noting $\Delta_{n,1}=\Delta_{n,n-1}=I-E$, we obtain that $\cpsi(\Delta_{n,1})\geq p$ and $\cpsi(\Delta_{n,n-1})\geq p$. 
 By induction method, we have 
 $\cpsi(\Delta_{n+1,r})\geq \frac{p}{2}$. Indeed, let us assume that $\cpsi(\Delta_{n,r})\geq \frac{2}{p}$, then 
 $$\frac{1}{\cpsi(\Delta_{n+1,r})}\leq \frac{2}{p}-\left(\frac{1}{n}-\frac{1}{n+1}\right)\frac{2}{p}\leq\frac{2}{p}.$$
 The argument remains true for $\clsi$.
\end{proof}
\bibliographystyle{alpha}

\begin{thebibliography}{MPV16}

\bibitem[BCR20]{BCR20}
Ivan Bardet, Angela Capel, and Cambyse Rouz{\'e}.
\newblock Approximate tensorization of the relative entropy for noncommuting
  conditional expectations.
\newblock {\em arXiv preprint arXiv:2001.07981}, 2020.

\bibitem[Bec89]{Beck89}
William Beckner.
\newblock A generalized poincar{\'e} inequality for gaussian measures.
\newblock {\em Proceedings of the American Mathematical Society}, pages
  397--400, 1989.

\bibitem[BGJ20]{bgj20}
Michael Brannan, Li~Gao, and Marius Junge.
\newblock Complete logarithmic sobolev inequalities via ricci curvature bounded
  below.
\newblock {\em arXiv preprint arXiv:2007.06138}, 2020.

\bibitem[BGL13]{bgl13}
Dominique Bakry, Ivan Gentil, and Michel Ledoux.
\newblock {\em Analysis and geometry of Markov diffusion operators}, volume
  348.
\newblock Springer Science \& Business Media, 2013.

\bibitem[BR76]{BR}
Ola Bratteli and Derek~W. Robinson.
\newblock Unbounded derivations of von neumann algebras.
\newblock 25(2):139, 1976.

\bibitem[BS03]{bs3}
Mikhail~Sh Birman and Michael Solomyak.
\newblock Double operator integrals in a hilbert space.
\newblock {\em Integral equations and operator theory}, 47(2):131--168, 2003.

\bibitem[BT06]{bt06}
Sergey~G Bobkov and Prasad Tetali.
\newblock Modified logarithmic sobolev inequalities in discrete settings.
\newblock {\em Journal of Theoretical Probability}, 19(2):289--336, 2006.

\bibitem[CM17]{CM}
Eric~A Carlen and Jan Maas.
\newblock Gradient flow and entropy inequalities for quantum markov semigroups
  with detailed balance.
\newblock {\em Journal of Functional Analysis}, 273(5):1810--1869, 2017.

\bibitem[CM20]{CM2}
Eric~A Carlen and Jan Maas.
\newblock Non-commutative calculus, optimal transport and functional
  inequalities in dissipative quantum systems.
\newblock {\em Journal of Statistical Physics}, 178(2):319--378, 2020.

\bibitem[Dem05]{demange}
J{\'e}r{\^o}me Demange.
\newblock Porous media equation and sobolev inequalities under negative
  curvature.
\newblock {\em Bulletin des sciences mathematiques}, 129(10):804--830, 2005.

\bibitem[DK51]{krein1}
Yu.\~L. Daleckii and S.~G. Krein.
\newblock Formulas of differentiation according to a parameter of functions of
  hermitian operators.
\newblock 76:13, 1951.

\bibitem[DPR17]{DPR17}
Nilanjana Datta, Yan Pautrat, and Cambyse Rouz{\'e}.
\newblock Contractivity properties of a quantum diffusion semigroup.
\newblock {\em Journal of Mathematical Physics}, 58(1):012205, 2017.

\bibitem[dPS04]{bs1}
B.~de~Pagter and F.~A. Sukochev.
\newblock Differentiation of operator functions in non-commutative
  $l_p$-spaces.
\newblock {\em Journal of Functional Analysis}, 212(1):28, 2004.

\bibitem[dPS07]{bs2}
Ben de~Pagter and Fyodor Sukochev.
\newblock Commutator estimates and $r$-flows in non-commutative operator
  spaces.
\newblock {\em Proceedings of the Edinburgh Mathematical Society. Series II},
  50(2):293, 2007.

\bibitem[DR20]{DR20}
Nilanjana Datta and Cambyse Rouz{\'e}.
\newblock Relating relative entropy, optimal transport and fisher information:
  A quantum hwi inequality.
\newblock In {\em Annales Henri Poincar{\'e}}, pages 1--36. Springer, 2020.

\bibitem[DS87]{ds87}
Persi Diaconis and Mehrdad Shahshahani.
\newblock Time to reach stationarity in the bernoulli--laplace diffusion model.
\newblock {\em SIAM Journal on Mathematical Analysis}, 18(1):208--218, 1987.

\bibitem[EMT15]{emt}
Matthias Erbar, Jan Maas, and Prasad Tetali.
\newblock Discrete ricci curvature bounds for bernoulli-laplace and random
  transposition models.
\newblock In {\em Annales de la Facult{\'e} des sciences de Toulouse:
  Math{\'e}matiques}, volume~24, pages 781--800, 2015.

\bibitem[FM16]{fm}
Max Fathi and Jan Maas.
\newblock Entropic ricci curvature bounds for discrete interacting systems.
\newblock {\em The Annals of Applied Probability}, 26(3):1774--1806, 2016.

\bibitem[GJL18]{LJR}
Li~Gao, Marius Junge, and Nicolas LaRacuente.
\newblock Fisher information and logarithmic sobolev inequality for matrix
  valued functions.
\newblock {\em arXiv preprint arXiv:1807.08838}, 2018.

\bibitem[Goe04]{goel}
Sharad Goel.
\newblock Modified logarithmic sobolev inequalities for some models of random
  walk.
\newblock {\em Stochastic processes and their applications}, 114(1):51--79,
  2004.

\bibitem[GQ03]{gao}
Fuqing Gao and Jeremy Quastel.
\newblock Exponential decay of entropy in the random transposition and
  bernoulli-laplace models.
\newblock {\em The Annals of Applied Probability}, 13(4):1591--1600, 2003.

\bibitem[Gro75]{gross}
Leonard Gross.
\newblock Logarithmic sobolev inequalities.
\newblock {\em American Journal of Mathematics}, 97(4):1061--1083, 1975.

\bibitem[GZ03]{gz03}
Alice Guionnet and B~Zegarlinksi.
\newblock Lectures on logarithmic sobolev inequalities.
\newblock In {\em S{\'e}minaire de probabilit{\'e}s XXXVI}, pages 1--134.
  Springer, 2003.

\bibitem[HN95]{SA}
S.~Hejazian and A.~Niknam.
\newblock Derivations of operator algebras, automatic closability.
\newblock In {\em Different aspects of differentiability}. 1995.

\bibitem[HP12]{HP12}
Fumio Hiai and D{\'e}nes Petz.
\newblock From quasi-entropy to various quantum information quantities.
\newblock {\em Publications of the Research Institute for Mathematical
  Sciences}, 48(3):525--542, 2012.

\bibitem[HP13]{HP}
Fumio Hiai and D{\'e}nes Petz.
\newblock Convexity of quasi-entropy type functions: Lieb's and ando's
  convexity theorems revisited.
\newblock {\em Journal of Mathematical Physics}, 54(6):062201, 2013.

\bibitem[HS86]{hs87}
Richard Holley and Daniel~W Stroock.
\newblock Logarithmic sobolev inequalities and stochastic ising models.
\newblock 1986.

\bibitem[IO80]{AS}
Atsushi Inoue and Sh{\^o}ichi Ota.
\newblock Derivations on algebras of unbounded operators.
\newblock {\em Transactions of the American Mathematical Society},
  261(2):567--577, 1980.

\bibitem[JRS14]{JRS18}
M.~Junge, E.~Ricard, and D.~Shlyakhtenko.
\newblock Noncommutative diffusion semigroups and free probability.
\newblock {\em Preprint}, 3, 2014.

\bibitem[Kap53]{Kap}
Irving Kaplansky.
\newblock Modules over operator algebras.
\newblock {\em American Journal of Mathematics}, 75(4):839--858, 1953.

\bibitem[kre56]{krein2}
Integration and differentiation of functions of hermitian operators and
  applications to the theory of perturbations.
\newblock 1956(1):81, 1956.

\bibitem[Led97]{ledoux97}
Michel Ledoux.
\newblock On talagrand's deviation inequalities for product measures.
\newblock {\em ESAIM: Probability and statistics}, 1:63--87, 1997.

\bibitem[Led99]{ledoux99}
Michel Ledoux.
\newblock Concentration of measure and logarithmic sobolev inequalities.
\newblock In {\em Seminaire de probabilites XXXIII}, pages 120--216. Springer,
  1999.

\bibitem[LJL20]{LJL}
Haojian Li, Marius Junge, and Nicholas LaRacuente.
\newblock Graph {H}\" ormander systems.
\newblock {\em arXiv preprint arXiv:2006.14578}, 2020.

\bibitem[LY98]{ly98}
Tzong-Yow Lee and Horng-Tzer Yau.
\newblock Logarithmic sobolev inequality for some models of random walks.
\newblock {\em The Annals of Probability}, 26(4):1855--1873, 1998.

\bibitem[MC89]{mc}
Elena~Aleksandrovna Morozova and Nikolai~Nikolaevich Chentsov.
\newblock Markov invariant geometry on state manifolds.
\newblock {\em Itogi Nauki i Tekhniki. Seriya" Sovremennye Problemy Matematiki.
  Noveishie Dostizheniya"}, 36:69--102, 1989.

\bibitem[MPV16]{mpv16}
Lajos Moln{\'a}r, J{\'o}zsef Pitrik, and D{\'a}niel Virosztek.
\newblock Maps on positive definite matrices preserving bregman and jensen
  divergences.
\newblock {\em Linear Algebra and its Applications}, 495:174--189, 2016.

\bibitem[Pet85]{Petz85}
D{\'e}nes Petz.
\newblock Quasi-entropies for states of a von neumann algebra.
\newblock {\em Publications of the Research Institute for Mathematical
  Sciences}, 21(4):787--800, 1985.

\bibitem[Pet96]{Petz96}
D{\'e}nes Petz.
\newblock Monotone metrics on matrix spaces.
\newblock {\em Linear algebra and its applications}, 244:81--96, 1996.

\bibitem[Pet07]{petz07}
D{\'e}nes Petz.
\newblock {\em Quantum information theory and quantum statistics}.
\newblock Springer Science \& Business Media, 2007.

\bibitem[Pet09]{Jesse}
Jesse Peterson.
\newblock A 1-cohomology characterization of property (t) in von neumann
  algebras.
\newblock {\em Pacific journal of mathematics}, 243(1):181--199, 2009.

\bibitem[PS10]{PS10}
D~Potapov and F~Sukochev.
\newblock Double operator integrals and submajorization.
\newblock {\em Mathematical Modelling of Natural Phenomena}, 5(4):317--339,
  2010.

\bibitem[PV15]{pv15}
J{\'o}zsef Pitrik and D{\'a}niel Virosztek.
\newblock On the joint convexity of the bregman divergence of matrices.
\newblock {\em Letters in Mathematical Physics}, 105(5):675--692, 2015.

\bibitem[RX16]{RX}
{\'E}ric Ricard and Quanhua Xu.
\newblock A noncommutative martingale convexity inequality.
\newblock {\em The Annals of Probability}, 44(2):867--882, 2016.

\bibitem[Sau90]{Sau}
Jean-Luc Sauvageot.
\newblock Quantum dirichlet forms, differential calculus and semigroups.
\newblock In {\em Quantum probability and applications V}, pages 334--346.
  Springer, 1990.

\bibitem[Spo78]{Spo78}
Herbert Spohn.
\newblock Entropy production for quantum dynamical semigroups.
\newblock {\em Journal of Mathematical Physics}, 19(5):1227--1230, 1978.

\bibitem[V{\'a}z07]{vazquez}
Juan~Luis V{\'a}zquez.
\newblock {\em The porous medium equation: mathematical theory}.
\newblock Oxford University Press, 2007.

\bibitem[Vir16]{vir16}
D{\'a}niel Virosztek.
\newblock Maps on quantum states preserving bregman and jensen divergences.
\newblock {\em Letters in Mathematical Physics}, 106(9):1217--1234, 2016.

\end{thebibliography}

\end{document}